\documentclass[12pt, twoside, reqno] {amsart}

\usepackage{amsfonts}
\usepackage{amssymb}
\usepackage{latexsym}
\usepackage{epsf,graphicx}
\usepackage{epsf,graphicx,latexsym,%
}

\newcommand{\be} {\begin{eqnarray}}
\newcommand{\ee} {\end{eqnarray}}
\newcommand{\bep} {\begin{eqnarray*}}
\newcommand{\eep} {\end{eqnarray*}}

\newcommand {\Hol}{\mathop{\rm Hol}\nolimits}

\newcommand {\Id}{\mathop{\rm Id}\nolimits}

\newcommand {\Lip}{\mathop{\rm Lip}\nolimits}

\newcommand {\dist}{\mathop{\rm dist}\nolimits}

\newcommand {\DD}{\mathcal{D}}
\newcommand {\Ff}{\mathcal{F}}

\newcommand{\R}{{\mathbb R}}
\newcommand{\N}{{\mathbb N}}

\newcommand{\C}{{\mathbb C}}

\newtheorem{remar}{Remark}[section]
\newtheorem{examp}{Example}[section]
\newtheorem{defin}{Definition}[section]
\newtheorem{corol}{Corollary}[section]
\newtheorem{propo}{Proposition}[section]
\newtheorem{theorem}{Theorem}[section]
\newtheorem{lemma}{Lemma}[section]
\newtheorem{remark}{Remark}[section]
\newtheorem{conj}{Conjecture}

\newcommand{\rema}{\begin{remar}\rm}
\newcommand{\erema}{$\blacktriangleright$\end{remar}}

\newcommand{\exa}{\begin{examp}\rm}
\newcommand{\eexa}{$\blacktriangleright$\end{examp}}

\def\lwvec(#1 #2){\linewd 0.1
           \lvec(#1 #2)
           \linewd 0.05}

\title{Differentiability of semigroups of Lipschitz or smooth mappings}

\author[M. Elin]{Mark Elin}

\address{Department of Mathematics,
         Ort Braude College,
         Karmiel 21982,
         Israel}

\email{mark$\_$elin@braude.ac.il}

\begin{document}

\begin{abstract}
In this note we study the differentiability with respect to the time-parameter of semigroups consisting of Lipschitzian or smooth self-mappings of a domain in a Banach space.

\end{abstract}

\maketitle

\section{Introduction}\label{sect-intro}
\setcounter{equation}{0}

In this paper we consider the following problem:

Let $\{F_t\}_{t\ge0}$ be a one-parameter semigroup consisting of continuous  self-mappings of a domain $\DD$  in a  Banach space $X$. Assume that this semigroup depends on the time-parameter $t$ continuously. The following question is fundamental:
\begin{center}
  {\it What conditions on the semigroup elements\\ entail its differentiability with respect to $t$?}
\end{center}
Indeed, solutions of autonomous dynamical systems (with identity initial data) known to be semigroups. In the reverse direction, if a semigroup is differentiable with respect to its parameter, it may satisfy an autonomous differential equation. At the same time, it is not clear whether a given semigroup is differentiable. Therefore the above question is relevant.

The following result is well-known and can be found in various books and textbooks (see, for example, \cite{D-Sch1958}).

\begin{theorem}\label{th_classic}
  Let $\{T_t\}_{t\ge0}$ be a strongly continuous semigroup of linear operators on $\DD=X$.

  \begin{itemize}
    \item [(a)] The mapping $t\mapsto T_tx$ is differentiable with respect to $t$ whenever $x$ belongs to a dense subspace of $X$.
    \item [(b)] It is differentiable for all $x\in X$ if and only if the semigroup  $\{T_t\}_{t\ge0}$ is uniformly continuous.
  \end{itemize}
\end{theorem}

The first assertion of this theorem was generalized  to semigroups of nonlinear contractions defined on nonexpansive retracts of $X$ in \cite{Kom} and \cite{R-81} and to semigroups of nonlinear operators on closed convex sets in some special Banach spaces in \cite{R-80}; see also \cite{R-84}. It was established in these works that a semigroup is differentiable with respect to the time-parameter whenever its spatial variable belongs to a dense subset of the set on which the semigroup is defined. This conclusion fails in general, that is, there are nonlinear semigroups that have not densely defined derivative with respect to $t$ (see, for example, \cite{B}).

Concerning generalizations of the second assertion, that is, the differentiability of semigroups of nonlinear operators everywhere on their domains, little is known. Some sufficient condition including H\"older continuity with respect to $t$ can be derived from \cite{W}. To the best of our knowledge,  this problem has been completely solved only for holomorphic mappings. Namely, in a series of papers \cite{R-S96}--\cite{R-S98}, Reich and Shoikhet proved that

\begin{theorem}\label{th_RS}
  A semigroup  $\{F_t\}_{t\ge0}\subset\Hol(\DD)$ of holomorphic self-mappings of a bounded domain ${\DD\subset X}$ is differentiable with respect to the parameter $t$ if and only if it is locally uniformly continuous {\rm (see Definition~\ref{def-cont} below)}. 
\end{theorem}

This generalizes the previous results by Berkson and Porta \cite{B-P} for the case where $\DD$ is the open unit disk in the complex plane $\C$ and by Abate \cite{A92} for the finite-dimensional case (see also \cite{R-S1, E-R-S-19} for more details).

The aim of this note is to find an analog of the Reich--Shoikhet theorem for semigroups of  not necessarily holomorphic  mappings. We introduce a subclass of Lipschitzian semigroups that are locally uniformly continuous and obtain a sufficient condition for differentiability of semigroups from this subclass in Section~\ref{sect_semigr}. Then we consider semigroups of smooth mappings in Section~\ref{sect_semigr1}.

\section{Main notions}\label{sect-defs}
\setcounter{equation}{0}

In this section we recall definitions of some notions used in the paper.  We start with some standard notations. Throughout the paper we denote by $X$ and $Y$ two (real or complex) Banach spaces
. Let $\DD\subset X$ and $\Omega\subset Y$ be domains (connected open sets). The set of all mappings that are continuous (respectively, smooth) on $\DD$ and take values in $\Omega$ is denoted by $C(\DD,\Omega)$ (respectively, $C^1(\DD,\Omega)$). If the Banach spaces $X$ and $Y$ are complex, a mapping $F:X\supset\DD\to\Omega\subset Y$ is said to be holomorphic if it is Fr\'{e}chet differentiable  at each point $x\in \DD$. By $\Hol(\DD,\Omega)$ we denote the set of all holomorphic mappings on $\DD$ with values in $\Omega$.

By $C(\DD)$ (respectively, $C^1(\DD)$ or $\Hol(\DD)$) we denote the set of all continuous (respectively, smooth or holomorphic) self-mappings of~$\DD$. Each one of these sets is  a semigroup with respect to composition operation. For $F\in C(\DD)$ we denote by $F^k$ the $k$-th iterate of $F$, that is, $F^1:= F$ and $F^{k+1}:=F\circ F^k,\ k\in\N$. The family of iterates $\{F^k\}_{k\in\N}$ forms a discrete-time semigroup on $\DD$. In this note we mostly concentrate on semigroups depending on continuous time (see Definition~\ref{def-sg} below).

Let a domain $\DD$ in a Banach space $ X$ be given. A bounded subset $\hat\DD \subset\DD$ is said to lie strictly inside $\DD$ if it is bounded away from the boundary $\partial\DD$ of $\DD$, that is, ${\inf\limits_{x\in\hat\DD} \dist(x,\partial \DD)>0}$. One of the surprising features of infinite-dimensional analysis is that a holomorphic mapping $f\in\Hol(\DD,Y)$ is not necessarily  bounded on  subsets lying strictly inside $\DD$ (see \cite{IJM-SLL-84, Har, R-S1}).

We will need some  different types of continuity for arbitrary families of mappings $\left\{f_t\right\}_{t\ge0}\subset C(\DD, Y)$ and relations between them.

\begin{defin}\label{def-cont}
The family $\left\{f_t\right\}_{t\ge0}\subset C(\DD,Y)$ is said to be
\begin{itemize}
\item jointly continuous (JC, for short) if for every $(t_0,x_0)\in[0,\infty)\times\DD$
\[
\lim_{t\to t_0,x\to x_0} f_t(x) =f_{t_0}(x_0);
\]

%
\item  locally uniformly continuous ($T$-continuous, for short) if for every $t_0\ge0$ and  for every subset $\hat\DD $  strictly inside~$\DD$,
\[
\sup_{x\in\hat\DD }\|f_t(x)-f_{t_0}(x) \| 
\to0\quad\mbox{as}\quad t\to t_0.
\]
\end{itemize}
\end{defin}
Notice that for the case where $X$ is finite-dimensional, 
local uniform continuity coincides with  uniform continuity on compact subsets. There are examples of JC families that are not 
T-continuous (see \cite{EJK-multi}).

We now define one-parameter continuous semigroups, which are the main object of interest  in this paper.

\begin{defin}\label{def-sg}
A jointly continuous family $\mathcal{F} =\left\{F_t\right\}_{t\ge0}\subset C(\DD)$ is called a one-parameter continuous semigroup (semigroup, for short) on $\DD$ if the following properties hold:

(i) $F_{t+s}=F_t \circ F_s$ for all $t,s\geq 0$;

(ii) $\lim\limits_{t\to 0^+}F_t(x) =x$ for all $x\in \mathcal{D}$.
\end{defin}
In particular, it can be shown that a semigroup is $T$-continuous if and only if $\left\| F_t(x) - x \right\|\to 0$  as $t\to0^+$ uniformly on every subset $\hat\DD$ strictly inside $\DD$.

Recall that the Lipschitz seminorm of $F\in C(\DD,Y)$ is defined by
\[
\Lip_\DD(F):=\sup\limits_{x,\tilde x\in\DD,x\not=\tilde x}\frac{\|F(x)-F(\tilde x)\|}{\|x-\tilde x\|}.
\]
The class of uniformly $k$-Lipschitzian semigroups, which is wider than the class of semigroups of nonexpansive mappings when $k>1$, was introduced in~\cite{G-K-T}. It consists of semigroups $\Ff =\left\{F_t\right\}_{t\ge0}\subset C(\DD)$ that satisfy $\Lip_\DD(F_t) \le k$ for all $t\ge0$. A Lipschitzian semigroup $\Ff$ is said to be uniformly continuous on $\DD$ if $\lim\limits_{t\to0^+} \Lip_\DD(F_t - \Id) = 0$ (see, for example,~\cite{P-X}), that is, if it is continuous with respect to the Lipschitz seminorm, uniformly  on~$\DD$. In what follows we will use a condition, which is weaker than Lipschitz continuity uniformly on subsets strictly inside $\DD$. To define it, for a subset  $\hat\DD $  strictly inside~$\DD$ with $\dist(\hat\DD, \partial\DD)> \mu>0$, we denote
\[
\Lip_{\hat\DD,\mu}(F):=\sup_{x\in\hat\DD,\|x-\tilde x\|\le\mu,x\not=\tilde x}   \frac{\|F(x)-F(\tilde x)  \|}{\|x-\tilde x\|}.
\]

Thus, the following notion is natural.
\begin{defin}\label{def-lip}
We say that a semigroup $\mathcal{F} =\left\{F_t\right\}_{t\ge0}\subset C(\DD)$ is locally uniformly Lipschitzian ($T$-Lipschitzian, for short) if for every subset $\hat\DD$  strictly inside~$\DD$ with $\dist(\hat\DD,\partial\DD)>\mu>0$,
\[
\Lip_{\hat\DD,\mu}(F_t-\Id)  \to0\quad\mbox{as}\quad t\to 0^+.
\]
\end{defin}
Note that if a Lipschitzian semigroup is uniformly continuous on subsets strictly inside $\DD$, then it is $T$-Lipschitzian. On the other hand, if a semigroup is $T$-Lipschitzian, then for every subset $\hat\DD$  strictly inside~$\DD$ and for every $\varepsilon>0$ there is $\delta>0$ such that the family $\left\{\left. F_t\right|_{\hat\DD}\right\}_{0\le t<\delta}$  is uniformly $(1+\varepsilon)$-Lipschitzian.

\section{$T$-Lipschitzian semigroups}\label{sect_semigr}
\setcounter{equation}{0}

In this section we establish  a sufficient condition for the differentiability of a $T$-Lipschitzian semigroup acting  on a domain $\DD\subset X$.

First let us mention that even in the one-dimensional case not all semigroups are differentiable with respect to $t$.
\begin{examp}
  Let $\DD=(-1,1)\subset\R$. Consider the family $\mathcal{F} =\left\{F_t\right\}_{t\ge0}\subset C(\DD)$ defined by
  \begin{eqnarray*}
    F_t(x) = \left\{  \begin{array}{ll}
                        e^{-t}x,                    & |x|>\frac12\,, 0\le t\le \ln(2|x|),\vspace{2mm} \\
                        2e^{-2t}x|x|,         & |x|>\frac12\,,  t> \ln(2|x|), \vspace{2mm}\\
                        e^{-2t}x,                  & |x|\le\frac12\,.
                      \end{array}                      \right.
  \end{eqnarray*}
  It can be easily verified that this family forms a semigroup. For every $x$ with $|x|>\frac12\,,$ this semigroup is not differentiable with respect to $t$ at $t_0= \ln(2|x|)$. Furthermore, this semigroup is not $T$-Lipschitzian.
\end{examp}

To prove our main results, we need the following auxiliary assertion that together with Lemma~\ref{lem_2} below gives us a non-holomorphic analog of Lemma~6.1 in \cite{R-S1}.

\begin{lemma}\label{lem_1}
Let $\DD\subset X$ be a domain and let $\hat \DD$ lie strictly inside $\DD$ and satisfy $ \dist(\hat \DD,\partial\DD)>\mu$  for some $\mu>0$. Denote $\DD_\mu:=\left\{ \tilde x\in\DD :\  \dist(\tilde x, \hat\DD)<\mu\right\}$. Let $p\in\N$. Let $\phi\in C(\DD)$ satisfy
 \begin{equation}\label{estim}
  \Lip_{\hat\DD,\mu}(\phi^k-\Id) \le\ell
  \end{equation}
 for all $k=1,\ldots,p,$ and
\begin{equation}\label{mu}
   \sup_{x\in\hat \DD}\left\| x-\phi(x) \right\| \le\mu.
  \end{equation}
Then for all $x\in\hat\DD$ we have
  \[
  \left\| x-\phi^p(x) -p(x-\phi(x))  \right\|  \le (p-1)\ell \|x-\phi(x)\|.
  \]
\end{lemma}
\begin{proof}
  First we note that \eqref{estim} means that for all $x\in\hat\DD$ and $x_*\in\DD $ such that $\|x-x_*\| \le\mu$ we have
    \begin{eqnarray*}
     && \left\| (x-\phi^k(x))  -(x_* - \phi^k(x_*) )   \right\|   \le \ell  \|x-x_*\|.
  \end{eqnarray*}
Since $\phi$ satisfies \eqref{mu}, we can take here $x_*=\phi(x)$   and then, using the triangle inequality, we get
  \begin{eqnarray*}
     \left\| x-\phi^p(x) -p(x-\phi(x))  \right\|  &=&  \left\|\sum_{k=0}^{p-1} \left[ \phi^k(x)-\phi^{k+1}(x) -(x-\phi(x)) \right]  \right\|   \\
     &\le &  \sum_{k=1}^{p-1} \left\| \left( \phi^k(x)-x\right) -\left(\phi^{k+1}(x) -\phi(x)\right)  \right\|  \\
     &\le&  \sum_{k=1}^{p-1}\ell\|x-\phi(x)\| ,
  \end{eqnarray*}
which completes the proof.
\end{proof}

Let now a semigroup $\Ff=\left\{F_t\right\}_{t\ge0}$ be given. With the goal of this paper  in mind, it is natural to denote
\begin{equation}\label{f_t}
  f_t(x):=\frac1t\left(F_t(x)-x\right)
\end{equation}
and to apply the lemma above to the mapping $\phi$ defined by $$\phi(x):= F_{t_0}(x) = t_0f_{t_0}(x)+x $$ for a fixed $t_0$.

\begin{corol}\label{cor1}
    Let  $\hat \DD$ lie strictly inside $\DD$ such that \ $ \dist(\hat \DD,\partial\DD)>\mu>0$. Denote $\DD_\mu:=\left\{ \tilde x:\  \dist(\tilde x, \hat\DD)<\mu\right\}\subset\DD$. Let $\Ff=\left\{F_t\right\}_{t\ge0}\subset C(\DD)$ and $f_t$ be defined by \eqref{f_t}. Let $p\in\N$ and $t_0>0$.  If
    \begin{equation}\label{estim1}
\Lip_{\hat\DD,\mu}(t_0kf_{t_0k})    \le\ell .
  \end{equation}
  for all $k=1,\ldots,p,$ and $\sup\limits_{x\in\hat \DD} \left\| f_{t_0}(x) \right\| \le\frac\mu{t_0}$,
then for all $x\in\hat\DD$ we have
   \[
  \left\| f_{pt_0}(x) -f_{t_0}(x)  \right\|  \le \frac{p-1}p \ell \|f_{t_0}(x)\|\le \ell \|f_{t_0}(x)\|.
  \]
\end{corol}

We are now ready to prove our main result.

\begin{theorem}\label{th-sg}
      Let $\Ff=\left\{F_t\right\}_{t\ge0}\subset C(\DD)$ be a $T$-continuous and $T$-Lipschitzian semigroup. Then the strong limit
 \[
f(x)=\lim_{t\to0^+} \frac1t\left(F_t(x)-x\right)
\]
    exists uniformly on subsets strictly inside $\DD$. Moreover, $f \in C(\DD, X)$ and is bounded on every set that lies strictly inside $\DD$. In turn, the mapping $u$ defined by $u(t,x)=F_t(x) ,\ (t,x)\in[0,\infty)\times\DD,$ solves the Cauchy problem
\begin{equation}\label{cauchy}
\left\{
\begin{array}{l}
\displaystyle \frac{d u(t,x)}{d t} =f(u(t,x)) \vspace{2mm} \\
u(0,x)=x\in\DD .
\end{array}%
\right.
\end{equation}
\end{theorem}

\begin{proof}
Let $\hat\DD$  be an arbitrary subset lying strictly inside $\DD$.  Take a positive number $\mu<\min\left(1, \dist(\hat\DD,\partial\DD)\right)$. Then the set $\DD_\mu:=\left\{ \tilde x:\  \dist(\tilde x, \hat\DD)< \mu\right\}$ lies strictly inside~$\DD$.   Since  the semigroup $\Ff$ is  $T$-continuous and $T$-Lipschitzian,  there is a positive $\delta_1$ such that
\begin{equation}\label{f_t-assum}
 \sup\limits_{x\in \DD_\mu} \left\|F_t(x)-x \right\| < \mu\quad\mbox{and}\quad \Lip_{\hat\DD,\mu}(F_t-\Id)  < \mu
 \end{equation}
 whenever $t\in[0,\delta_1 )$. These inequalities can be rewritten as
  \begin{equation}\label{f_t-assum1}
     \sup\limits_{x\in \DD_\mu} \left\|f_t(x) \right\| <\frac\mu t\qquad\mbox{and}\qquad \Lip_{\hat\DD,\mu}(f_t)   <\frac \mu t.
  \end{equation}
For every $t_0\in\left(0,\frac{\delta_1}2\right)$, the interval $\left(\frac{\delta_1}{2t_0}, \frac{\delta_1} {t_0} \right)$ contains at least one natural number $p\ge2$. Since $\hat\DD\subset\DD_\mu$, inequalities \eqref{f_t-assum1} and the choice of $p$ mean that the assumptions of Corollary~\ref{cor1} are satisfied with $\ell=\mu$. Therefore,
 \[
    \left\| f_{t_0}(x)  \right\|  -   \left\| f_{pt_0}(x)   \right\| \le     \left\| f_{pt_0}(x) -f_{t_0}(x)  \right\|  \le \mu\|f_{t_0}(x)\|,\qquad x\in\hat\DD,
  \]
hence  $ (1-\mu)   \left\| f_{t_0}(x)  \right\| \le  \left\| f_{pt_0}(x)   \right\|  <\frac{\mu}{pt_0}  <\frac{2\mu}{\delta_1} \, $ by   \eqref{f_t-assum1}. Thus \begin{equation}\label{L}
  \left\| f_{t}(x)  \right\|< \frac{2\mu}{(1-\mu)\delta_1}=:L\quad \mbox{for all}\quad x\in \hat\DD\quad \mbox{and}\quad t\in\left(0,\frac{\delta_1}2\right).
\end{equation}

As above, since  $\Ff$ is  $T$-continuous and $T$-Lipschitzian, we conclude that for any positive $\varepsilon<\mu $ there is a positive number $\delta<\min\left\{\frac{\delta_1}2\,, \varepsilon \right\}$ such that inequalities \eqref{f_t-assum1} with $\mu$ replaced by $\varepsilon$ hold whenever $t<\delta$. Using this remark, we now show that the family  $\left\{f_t\right\}_{t>0} $ satisfies the Cauchy criterion as $t\to0^+$.

For $s,t>0$ denote $m=\left[\frac1s +1\right],\ n=\left[\frac1t +1\right]$. Then the triangle inequality gives us
\begin{eqnarray*}
  \left\|f_s(x)-f_t(x)  \right\|  &\le&  \left\|f_{\frac1{n}}(x)- f_{\frac1{m}}(x)  \right\| + \left\|f_s(x)- f_{\frac1{m}}(x)  \right\| \\
  &+& \left\|f_t(x)- f_{\frac1{n}}(x)  \right\|   = A_1+A_2+A_3,
\end{eqnarray*}
where the notations $A_1,A_2,A_3$ are evident.

If $s,t<\delta$, then   $m,n>\frac1\delta$\,. So,  we can choose $t_0=\frac1{mn},\ \mu=\ell=\varepsilon$ in Corollary~\ref{cor1} and apply it with  $p=m$ and $p=n$. This and \eqref{L} lead us to the   estimate:
\begin{eqnarray*}
 A_1  &\le&  \left\|f_{\frac1n}(x)- f_{\frac1{mn}}(x)  \right\| +  \left\|f_{\frac1{mn}}(x)- f_{\frac1m}(x)  \right\|\\
   &\le& \frac{m-1}{m}\,\varepsilon  \left\|f_{\frac1{mn}}(x) \right\| + \frac{n-1}{n}\,\varepsilon  \left\|f_{\frac1{mn}}(x) \right\| \\
   &\le& 2 \varepsilon  \left\|f_{\frac1{mn}}(x) \right\| \le 2\varepsilon L\qquad\mbox{for all}\qquad x\in\hat\DD.
\end{eqnarray*}

To estimate $A_2$, we rewrite the difference within the norm in the form
\begin{eqnarray*}
 && f_s(x)- f_\frac1m(x)    =  \frac1s \left( F_s(x)-x\right) - m\left(F_{\frac1m}(x)-x\right)  \\
     &=&  \frac1s \left( F_s(x)-F_{\frac1m}(x)\right) - \left(m-\frac1s\right)\left(F_{\frac1m}(x)-x\right)   \\
     &=& \frac{sm-1}{sm}\left[ \frac{1}{s-\frac1m} \left( F_{s-\frac1m}(F_{\frac1m}(x))-F_{\frac1m}(x)\right) - m\left(F_{\frac1m}(x)-x\right) \right] \\
     &=& \frac{sm-1}{sm}\left[ f_{s-\frac1m} \left(F_{\frac1m}(x)\right) - f_{\frac1m}(x)\right].
\end{eqnarray*}
Therefore
\[
  A_2    \le
  \frac{sm-1}{sm}\left( \left\| f_{s-\frac1m} \left(F_{\frac1m}(x)\right)\right\| +\left\| f_{\frac1m}(x) \right\| \right)   .
\]

Further, $\frac1m<\delta<\delta_1$, hence $\left\|F_\frac1m (x)-x\right\|<\mu$ for all $x\in\DD_\mu$ by \eqref{f_t-assum}. In particular, $F_\frac1m (x)\in\DD_\mu$ for all $x\in\hat\DD$. Therefore, \eqref{L} can be applied at both points $x$ and $F_{\frac1m}(x)$. So, we get
\[
\left\| f_{s-\frac1m} \left(F_{\frac1m}(x)\right)\right\| +\left\| f_{\frac1m}(x)\right\|\le 2L\quad \mbox{for all}\quad x\in\hat\DD.
\]
 Since $0< \frac{sm-1}{sm}< s$, we conclude that $A_2$  is less than~$s\cdot2L<2\delta L<2\varepsilon L$ for all $x\in\hat\DD$.

Similarly, $A_3<2\varepsilon L$ for all $x\in\hat\DD$. Summarizing, we conclude that for any $\varepsilon>0$ there is a number $\delta>0$ such that $\sup\limits_{x\in\hat\DD} \left\|f_s(x)-f_t(x)  \right\|  <6\varepsilon L$ for all $s$ and $t$ less than $\delta$.

Thus $\left\{f_t\right\}_{t>0} \subset C(\hat\DD,X)$ is a Cauchy net, so it converges uniformly on~$\hat\DD$ as $t\to0^+$. Hence its limit $f$ is continuous on $\hat\DD$. Moreover, by \eqref{L}  this mapping  is bounded on $\hat\DD$. Since the subdomain $\hat\DD$ strictly inside $\DD$ was chosen arbitrarily, we conclude that $f\in C(\DD,X)$.

In fact, we have shown that for every $x\in\DD$, $F_t(x)$ is right-differentiable  at $t=0$. For $t>0$ this conclusion follows by the semigroup property:
\[
\lim_{s\to0^+} \frac1s\left(F_{t+s}(x)-F_t(x) \right) = \lim_{s\to0^+} \frac1s\left(F_{s}(F_t(x))-F_t(x) \right) =f(F_t(x)).
\]
 The proof is complete.
\end{proof}

\begin{remark}
  It can be easily seen that if the mapping $f$  defined by the strong limit $f(x):=\lim\limits_{t\to0^+} \frac1t(F_t(x)-x)$ is bounded on subsets that lie strictly inside $\DD$, then $\Ff$ is $T$-continuous.
  \end{remark}

\section{Semigroups of smooth mappings}\label{sect_semigr1}
\setcounter{equation}{0}

We now concentrate on semigroups consisting of smooth self-mappings of a domain $\DD\subset X$. We start with the following simple assertion (cf. Lemma~\ref{lem_1}).
\begin{lemma}\label{lem_2}
Let $\DD\subset X$ be a domain and let $\hat \DD$ lie strictly inside $\DD$ and satisfy $ \dist(\hat \DD,\partial\DD)>\mu$  for some $\mu>0$. Denote $\DD_\mu:=\left\{ \tilde x\in\DD :\  \dist(\tilde x, \hat\DD)<\mu\right\}$.    If $\phi\in C^1(\DD)$ satisfies
  \begin{equation}\label{ell}
 \sup_{x\in \DD_\mu}  \left\|\Id_X - \phi'(x)\right\| \le \ell ,
  \end{equation}
  then $ \Lip_{\hat\DD,\mu}(\phi-\Id) \le\ell$.
 \end{lemma}
\begin{proof}
  Let $x\in\hat\DD$ and $x_*\in\DD $ be such that $\|x-x_*\| \le\mu$. Then thanks to \eqref{ell}, we have for $k=1,\ldots,p,$
  \begin{eqnarray*}
     && \left\| (x-\phi(x))  -(x_* - \phi(x_*) )   \right\|  \\
     =&&   \left\|  \int_0^1\left[ (x-x_*) - [\phi'(tx_*+(1-t)x)] (x-x_*)\right] dt \right\|           \\
    \le &&  \|x-x_*\| \cdot  \int_0^1\left\|\Id_X  - \phi'(tx_*+(1-t)x)  \right\| dt \le \ell  \|x-x_*\| ,
  \end{eqnarray*}
 so the result follows.
\end{proof}

\begin{theorem}\label{th-sg1}
      Let $\Ff=\left\{F_t\right\}_{t\ge0}\subset C^1(\DD)$ be a $T$-continuous semigroup such that the family of the  Fr\'{e}chet  derivatives $\left\{{F_t}'\right\}_{t\ge0}$ is $T$-continuous. Then the strong limit $f(x)=\lim\limits_{t\to0^+} \frac1t\left(F_t(x)-x\right)$   exists uniformly on subsets strictly inside $\DD$.
\end{theorem}

\begin{proof}
  This theorem will follow immediately from Theorem~\ref{th-sg} if we show that  the semigroup $\Ff$ is $T$-Lipschitzian.

  Let $\hat\DD$  be an arbitrary subset lying strictly inside $\DD$.  Take a positive number $\mu< \dist(\hat\DD,\partial\DD)$. Then the set $\DD_\mu:=\left\{ \tilde x:\  \dist(\tilde x, \hat\DD)< \mu\right\}$ lies strictly inside~$\DD$.   Since  the family $\left\{{F_t}'\right\}_{t\ge0}$ is $T$-continuous, for any positive $\varepsilon<\mu$ there is a positive $\delta$ such that
\[
  \sup\limits_{x\in\DD_\mu} \left\|(F_t)'(x)-\Id_X \right\| < \varepsilon
 \]
 whenever $t\in[0,\delta )$. Then $\Lip_{\hat\DD,\mu}(F_t-\Id) \le\varepsilon$ by Lemma~\ref{lem_2}. The proof is complete.
 \end{proof}

In the holomorphic case, $T$- continuity of a semigroup implies {$T$-continuity} of the family of the Fr\'{e}chet derivatives thanks to  the Cauchy inequality.  Therefore we get the following consequence.

\begin{corol}\label{cor3}
   Let $\DD$ be a domain in a complex Banach space. Let $\Ff=\left\{F_t\right\}_{t\ge0}\subset\Hol(\DD)$ be a $T$-continuous  semigroup.  Then $\Ff$ is differentiable with respect to $t$.
\end{corol}
This coincides with the sufficient (but not the necessary) part of the theorem by Reich and Shoikhet (Theorem~\ref{th_RS}).

It is interesting to ask about the reverse implication:  whether $T$-continuity of the family consisting of the  Fr\'{e}chet  derivatives can imply  $T$-continuity of the family of mappings themselves. It turns out that such conclusion depends on the geometry of the domain $\DD$. We now define the required geometric property.

\begin{defin}\label{def-paths}
We say that a domain $\DD\subset X$ has the  finite path-length property if for every subdomain $\DD_1$ which lies strictly inside $\DD$, there is another domain $\DD_2\supset\DD_1$, which also  lies strictly inside $\DD$ and a number $L$ such that for every pair of points $x_1,x_2\in\DD_2$ there is a smooth curve joining these points that lies in $\DD_2$ and has length less than $L$.
\end{defin}

It can be easily shown that if $\DD$ is a finite union of bounded convex domains or the space $X,\ X\supset\DD,$ is finite-dimensional, then the domain $\DD$  has this property. The following simple example shoes that in general there are domains that have not  the  finite path-length property.

\begin{examp}
Let consider the Banach space $\ell^\infty$. For $a\in\left(0,\frac12\right)$, denote by $D_j^{a}$ the following domains: $$D_1^a:=\left\{ x\in\ell^\infty:\ \left|x_1-\frac12\right|<\frac12+a, \  |x_k|<a \ \mbox{for}\ k>1 \right\}$$ and for $j>1,$
\[
D_j^a:=\left\{ x\in\ell^\infty:\ \begin{array}{lc}
                                   |x_k-1|<a &  \mbox{for}\ k<j, \vspace{2mm}\\
                                   \left|x_j-\frac12\right|<\frac12+a, &  \vspace{2mm} \\
                                  |x_k|<a & \mbox{for}\ k>j
                                 \end{array}\right\}.
\]

Let now $\DD:=\cup_{j=1}^\infty D_j^{1/3} $. This is a domain since $D_j^a\cap D_{j+1}^a\not=\emptyset$. This domain is bounded, namely, it is contained in the ball of radius $\frac43$. Take its subdomain $\DD_1:=\cup_{j=1}^\infty D_j^{1/10}$. It lies strictly inside~$\DD$.

Choose now the points $x^{(j)}\in\DD_1$ as follows: $x^{(j)}_k=1$ when $k\le j$ and $x^{(j)}_k=0$ when $k>j$. The length of any curve in $\DD$ joining the point  $x^{(j)}$ with the origin is greater than $j/2$, that is,  tends to infinity as $j\to\infty$. Thus, $\DD\subset X$ does not have the   finite path-length property.
\end{examp}

Our interest in the  finite path-length property is based on the next assertion.

\begin{propo}\label{propo_T-cont}
  Let a domain $\DD\subset X$ have the  finite path-length property and let $\left\{f_t\right\}_{t\ge0}\subset C^1(\DD,Y)$ be a jointly continuous family. If the family  $\left\{{f_t}'\right\}_{t\ge0}\subset C(\DD,L(Y))$ consisting of the Fr\'{e}chet derivatives is $T$-continuous, then the family  $\left\{f_t\right\}_{t\ge0}$ is $T$-continuous too.
\end{propo}
\begin{proof}
 Let domain $\DD_1$ lie strictly inside $\DD$. Fix a point $x_1\in\DD_1$. For any $t_0$ and any $\varepsilon>0$ there is $\delta_1>0$ such that $\|f_t(x_1)-f_{t_0}(x_1) \| _Y<\varepsilon$ as $|t-t_0|<\delta_1$.

 By Definition~\ref{def-paths}, there is a domain $\DD_2\supset\DD_1$, which also  lies strictly inside $\DD$ and a number $L$ such that one can join $x_1$ with every point $x_2\in\DD_2$ by a smooth curve  $x(q):[0,1]\to\DD_2$  of length less than $L$.  Then we have:
  \begin{eqnarray*}
       && \|f_t(x_1)-f_{t_0}(x_1)-f_t(x_2) +f_{t_0}(x_2)\| _Y \\
       &\le& \left\|\int_0^1 \left( (f_t)'(x(q))-(f_{t_0})'(x(q)) \right) x'(q)dq \right\| _Y      \\
       &\le&\int_0^1 \left\| (f_t)'(x(q))-(f_{t_0})'(x(q)) \right\| _{L(Y)} \left\|x'(q)\right\| _Y dq  \\
       &\le& L\cdot \sup_{x\in\DD_2} \left\| (f_t)'(x)-(f_{t_0})'(x) \right\| _{L(Y)}.
 \end{eqnarray*}

According to Definition~\ref{def-cont},  there is $\delta_2>0$ such that for all $x\in\DD_2$ we have ${\left\|(f_t)'(x)-(f_{t_0})'(x) \right\| _{L(Y)}<\varepsilon}$ whenever $|t-t_0|<\delta_2$.  Therefore, if $|t-t_0|<\min\{\delta_1,\delta_2\}$, then
 \begin{eqnarray*}
     &&  \|f_t(x_2)-f_{t_0}(x_2) \| _Y  \\
     &\le& \|f_t(x_1)-f_{t_0}(x_1)-f_t(x_2) +f_{t_0}(x_2)\| _Y + \|f_t(x_1)-f_{t_0}(x_1) \| _Y \\
    &\le& (L+1)\varepsilon.
 \end{eqnarray*}
 This estimate means that $f_t(x)\to f_{t_0}(x)$ as $t\to t_0$, uniformly on  $\DD_1$.
\end{proof}

This proposition together with the remark that if $X$ is finite-dimensional, the joint continuity of the family $\left\{{F_t}'\right\}_{t\ge0}$ implies its $T$-continuity by compactness, allow us to the following consequence of Theorem~\ref{th-sg1}.

\begin{corol}\label{cor2}
       Let $\Ff=\left\{F_t\right\}_{t\ge0}\subset C^1(\DD)$ be a semigroup on a domain $\DD\subset X$.  If $\DD$  has  the finite path-length property and the family of the  Fr\'{e}chet  derivatives $\left\{{F_t}'\right\}_{t\ge0}$ is $T$-continuous, then $\Ff$  is differentiable with respect to $t$. In particular, this holds when  $X$ is finite-dimensional Banach space and the family of the  Fr\'{e}chet  derivatives $\left\{{F_t}'\right\}_{t\ge0}$ is JC.
\end{corol}

Additionally, in the connection with Theorem~\ref{th-sg1} we mention  the classical work by Bochner and Montgomery \cite{Bo-Mon} from which follows that {\it a group of transformations of a finite-dimensional manifold $M$, which is of class $C^k$  for each $t$, is generated by a $C^{k-1}$ vector field} (see also \cite{Ch-M}). Hence the following conjecture seems to be natural.

\begin{conj}
  Let $\Ff=\left\{F_t\right\}_{t\ge0}\subset C^k(\DD)$ be a $T$-continuous semigroup such that for  each $j=1,\ldots,k$ the  family $\left\{{F_t}^{(j)}\right\}_{t\ge0}$ is $T$-continuous. Then the strong limit
 \[
f(x)=\lim_{t\to0^+} \frac1t\left(F_t(x)-x\right)
\]
    exists uniformly on subsets strictly inside $\DD$. Moreover, $f $ is bounded on every set that lies strictly inside $\DD$ and belongs to the class $C^{k-1}(\DD, X)$.
\end{conj}

\medskip

{\bf Acknowledgments.} The author is grateful to Guy Katriel for very fruitful discussions.  

\medskip



\begin{thebibliography}{950}


\bibitem{A92} M. Abate,
The infinitesimal generators of semigroups of holomorphic maps, {\it Ann. Mat. Pura Appl.} {\bf 161} (1992), 167--180.

\bibitem{B}  V. Barbu,
{\sl Nonlinear Semigroups and Differential Equations in Banach Spaces}, Noordhoff, Leyden, 1976.

\bibitem{B-P} E. Berkson and H. Porta,
Semigroups of analytic functions and composition operators, {\it Michigan Math. J.} {\bf 25} (1978), 101--115.


\bibitem{Bo-Mon} S. Bochner and D. Montgomery,
Groups of differentiable and real or complex analytic transformations, {\it Ann. Math.} {\bf 46} (1945), 685--694.

\bibitem{Ch-M}  P. Chernoff, and J. Marsden,
On continuity and smoothness of group actions, {\it Bull. Amer. Math. Soc.} {\bf 76} (1970), 1044--1049.





%
\bibitem{D-Sch1958} N. Dunford and J. T. Schwartz,
{ \sl Linear Operators}, Vol. I, Interscience, New York, 1958.





\bibitem{EJK-multi}  M. Elin, F. Jacobzon, G. Katriel,   Continuous and holomorphic semicocycles in Banach spaces, {\it J. Evol. Equ.} {\bf 19} (2019), 1199--1221. 




\bibitem{E-R-S-19} M. Elin, S. Reich and D. Shoikhet,
{\sl Numerical range of holomorphic mappings and applications}, Birkh\"{a}user, Cham, 2019.


\bibitem{G-K-T} K. Goebel, W. A. Kirk, and R. L. Thele,
Uniformly lipschitzian families of transformations in Banach spaces, {\it Can. J. Math.} {\bf 26} (1974), 1245--1256.


\bibitem{IJM-SLL-84} J. M. Isidro and L. L. Stacho,
{\sl Holomorphic Automorphism Groups in Banach Spaces: An
Elementary Introduction}, North Holland, Amsterdam, 1984.

\bibitem{Har} L. A. Harris,
The numerical range of holomorphic functions in Banach spaces,   \textit{Amer. J. Math.} {\bf 93}  (1971), 1005--1019.




\bibitem{Kom} Y. K\={o}mura,
Differentiability of nonlinear semigroups, {\it  J. Math. Sot. Japan} {\bf 21} (1969), 375--402.




\bibitem{P-X}  J. G. Peng and  Z. B. Xu,
A novel dual approach to nonlinear semigroups of Lipschitz operators, {\it Trans. Amer. Math. Soc.} {\bf 357} (2005) 409--424.


\bibitem{R-80}  S. Reich,
Product formulas, nonlinear semigroups, and accretive operators, {\it J. Functional Anal.} {\bf 36} (1980), 147--168.


\bibitem{R-81}   S. Reich,
A nonlinear Hille-Yosida theorem in Banach spaces, {\it J. Math. Anal. Appl.} {\bf 84} (1981), 1--5.


\bibitem{R-84}  S. Reich,
On the differentiability of nonlinear semigroups. {\it Infinite-dimensional systems (Retzhof, 1983)}, 203--208,
Lecture Notes in Math., 1076, Springer, Berlin, 1984.



\bibitem{R-S96} S. Reich and D. Shoikhet,
Generation theory for semigroups of holomorphic mappings in Banach spaces, \textit{Abstr. Appl. Anal.} {\bf 1}  (1996), 1--44.


\bibitem{R-S97} S. Reich and D. Shoikhet,
Semigroups and generators on convex domains with the hyperbolic metric, {\it Atti Accad. Naz. Lincei} {\bf 8} (1997), 231--258.


\bibitem{R-S98} S. Reich and D. Shoikhet,
Metric domains, holomorphic mappings and nonlinear semigroups,  {\it Abstract Appl. Anal.} {\bf 3} (1998), 203--228.


\bibitem{R-S1} S. Reich and D. Shoikhet,
{\sl Nonlinear Semigroups, Fixed Points, and the Geometry of
Domains in Banach Spaces}, World Scientific Publisher, Imperial
College Press, London, 2005.

\bibitem{W} L. Waelbroeck,
 Differentiability of H\"{o}lder-continuous semigroups, {\it Proc. Amer. Math. Soc.} {\bf 21} (1969),  451--454.


\end{thebibliography}
\end{document}